\documentclass[reqno,a4paper]{amsart}

\usepackage{amssymb,amsfonts,amsmath,amsthm}
\usepackage{mathrsfs}
\usepackage{graphicx}
\usepackage{color}
\usepackage{verbatim}
\usepackage{epigraph}

\newtheorem{thm}{Theorem}[section]
\newtheorem{pro}[thm]{Proposition}
\newtheorem{lem}[thm]{Lemma}
\newtheorem{cor}[thm]{Corollary}

\theoremstyle{definition}

\newtheorem{exa}[thm]{Example}

\theoremstyle{remark}
\newtheorem{rmk}[thm]{Remark}

\numberwithin{equation}{section}

\def\J{\mathscr{J}}
\def\D{\mathscr{D}}
\def\R{\mathscr{R}}
\def\L{\mathscr{L}}
\def\H{\mathscr{H}}
\def\cE{\mathcal{E}}
\def\T{\mathcal{T}}
\def\eps{\varepsilon}
\def\es{\varnothing}
\def\Ra{\Rightarrow}

\def\ol#1{\overline{#1}}

\def\ig#1{\mathsf{IG}(#1)}

\def\gen#1{\langle #1 \rangle}
\def\pre#1#2{\langle #1 \; | \; #2 \rangle}

 \DeclareMathOperator\id{id}
\DeclareMathOperator\im{im} 
 \DeclareMathOperator\rank{rank}
\DeclareMathOperator\Stab{Stab} \DeclareMathOperator\AHom{AHom}

\begin{document}


\title[Free idempotent-generated semigroups]%
{Elaborating the word problem for \\ free idempotent-generated semigroups \\ over the full transformation monoid} 


\author{IGOR DOLINKA}

\address{Department of Mathematics and Informatics, University of Novi Sad, Trg Dositeja Obradovi\'ca 4,
21101 Novi Sad, Serbia}

\email{dockie@dmi.uns.ac.rs}

\thanks{This research was supported by the Ministry of Education, Science, and Technological Development of the Republic of Serbia.}








\begin{abstract}
With each semigroup one can associate a partial algebra, called the biordered set, which captures important algebraic and 
geometric features of the structure of idempotents of that semigroup. For a biordered set $\cE$, one can
construct the free idempotent-generated semigroup over $\cE$, $\ig{\cE}$, which is the free-est semigroup (in a definite
categorical sense) whose biorder of idempotents is isomorphic to $\cE$. Studies of these intriguing objects have been
recently focusing on their particular aspects, such as maximal subgroups, the word problem, etc. In 2012, Gray and Ru\v skuc
pointed out that a more detailed investigation into the structure of the free idempotent-generated semigroup over the biorder 
of $\T_n$, the full transformation monoid over an $n$-element set, might be worth pursuing. In 2019, together with Gould and
Yang, the present author showed that the word problem for $\ig{\cE_{\T_n}}$ is algorithmically soluble. In a recent work 
by the author, it was showed that, for a wide class of biorders $\cE$, the algorithmic solution of the word problem revolves 
around the so-called vertex groups, which arise as certain subgroups of direct products of pairs of maximal subgroups
of $\ig{\cE}$. In this paper we determine these vertex groups for the case when $\cE$ is the biorder of idempotents of $\T_n$.
\end{abstract}


\maketitle

\epigraph{\textit{Voici mon secret. Il est tr\`es simple: on ne voit bien qu'avec le c{\oe}ur. L'essentiel est invisible pour les
yeux.} [...] \textit{C'est le temps que tu as perdu pour ta rose qui fait ta rose si importante.}}{--- A. de Saint-Exup\'ery, 
\textit{Le petit prince}}


\section{Introduction}\label{sec:intro}

For a semigroup $S$, let $E(S)$ denote the set of its idempotents. However, merely recording the collection of idempotent
elements of a semigroup is frequently not enough, as a significant amount of information about the mutual
relationships of these elements, as well as their impact to the general structure, is lost in this way. Therefore, as it 
turns out, it is useful to consider a partial algebra $\cE_S=(E,\cdot)$, where $E=E(S)$, obtained by retaining products of
idempotents $ef$ such that $\{ef,fe\}\cap\{e,f\}\neq\es$. (This amounts to saying that the product of $e$ and $f$, multiplied
in some order, results in one of the factors. It is then easily verified that the product of $e$ and $f$ in the reverse
order is an idempotent, too, although not necessarily equal to one of $e,f$.) Such a pair $\{e,f\}$ is called a 
\emph{basic pair}.

The partial algebra $\cE_S$ obtained in this way is called the \emph{biordered set} of $S$. The name derives from the fact that
it is convenient to define two quasi-orders on $\cE_S$: namely, let $e\leq_\ell f$ if and only if $ef=e$, and let $e\leq_r f$
if and only if $ef=f$. This effectively captures the basic pairs of $S$; in addition, the intersection $\leq = \leq_\ell \cap
\leq_r$ is precisely the \emph{natural order} of idempotents of $S$ (see e.g.\ \cite{HoBook}). It was shown by Nambooripad \cite{Na} 
and Easdown \cite{Ea2} that these partial structures can be finitely axiomatised: there is a finite set of formul{\ae} such that 
for any abstract structure $\cE$ satisfying these axioms there is a semigroup $S$ such that $\cE\cong\cE_S$. Of course, the 
biordered set of any finite semigroup is finite, while the converse is not necessarily the case: there are finite biorders not 
stemming from any finite semigroup \cite{Ea1}.

Crucial in the study of idempotent-generated semigroups (semigroups $S$ with the property that $S=\gen{E(S)}$), a very natural
and omnipresent class of semigroups, is the notion of a \emph{free idempotent-generated semigroup on a biordered set} $\cE$.
It is defined by the presentation
$$
\ig{\cE} = \pre{\ol{E}}{\ol{e}\ol{f}=\ol{e\cdot f}\text{ whenever }\{e,f\}\text{ is a basic pair in }\cE},
$$
where $\ol{E}=\{\ol{e}:\ e\in E\}$ is an alphabet in a one-to-one correspondence with $E$, the set of elements of the biorder $\cE$.
This is, in a quite definite sense, the ``free-est'' idempotent-generated semigroup with biordered set isomorphic to $\cE$ (in the
case of $\ig{\cE}$, its biorder $\ol{\cE}$ is formed by elements of $\ol{E}$). More precisely, if $S$ is any semigroup such that 
$\cE_S\cong \cE$, with $\phi:\cE\to\cE_S$ being a biordered set isomorphism, then $\phi'=\iota^{-1}\phi:\ol{\cE}\to\cE_S$ (where 
$\iota:\cE\to\ol{\cE}$, defined by $e\iota=\ol{e}$, $e\in E$, is also a biorder isomorphism) can be (uniquely) extended to a semigroup 
homomorphism $\Psi_\phi:\ig{\cE}\to S$ (here the index $\phi$ intends to indicate that the homomorphism $\Psi$ depends on the choice 
of $\phi$; in other cases, when the initial isomorphism $\phi$ is irrelevant, we will just suppress this index). The image of this 
homomorphism is precisely the idempotent-generated part of $S$, namely its subsemigroup $S'=\gen{E(S)}$.

Free idempotent-generated semigroups were introduced by Nambooripad in \cite{Na} within a wider framework of a general study of regular
semigroups (see also \cite{NP,Pa}). Since then, they have been an object of fascination of an array of algebraists. The most recent 
resurgence of interest in this topic was initiated by papers \cite{BMM} and \cite{GR1}. Namely, for some time, a folklore conjecture 
(recorded officially only in \cite{McE}) was in circulation that the maximal subgroups of free idempotent-generated semigroups must 
necessarily be free groups. This conjecture proved to be wrong in a rather strong fashion: first, Brittenham, Margolis, and Meakin 
\cite{BMM} constructed a 73-element semigroup $S$ (containing 37 idempotents) such that $\ig{\cE_S}$ contains a maximal subgroup isomorphic 
to $\mathbb{Z}\times\mathbb{Z}$ (so, not a free group), and then, in a ground-breaking paper \cite{GR1}, Gray and Ru\v skuc showed that for 
\emph{any} group $G$ there is a suitable biorder $\cE$ such that $\ig{\cE}$ contains a maximal subgroup isomorphic to $G$. This was 
followed by a series of papers studying these maximal subgroups, see e.g.\ \cite{YDG1,Do1,DG,DR,GY,GR2}, after which the focus shifted 
to other structural features, and, primarily, to the question of the word problem. The pioneering paper in this sense were \cite{YG,DGR}, 
where the later exhibited the first example of a finite biorder $\cE$ (stemming, by the way, from a finite idempotent semigroup) such that 
all maximal subgroups of $\ig{\cE}$ have decidable word problems (in fact, they were all free or trivial) while the word problem for 
$\ig{\cE}$ is algorithmically unsolvable. The true nature of these problems was revealed in the papers \cite{YDG2} and \cite{Do2}, where 
it was shown that the word problem of $\ig{\cE}$ is in fact equivalent to a specific type of a constraint satisfaction problem related to 
certain subgroups (called the \emph{vertex groups} in \cite{Do2}) of direct products of pairs of maximal subgroups of $\ig{\cE}$. 

The aim of the present paper is to determine these groups for the free idempotent-generated semigroup over $\cE_{\mathcal{T}_n}$, the
biorder of the full transformation monoid $\mathcal{T}_n$ over an $n$-element set. We recall that the maximal subgroups of 
$\ig{\cE_{\mathcal{T}_n}}$ were previously computed in \cite{GR2}: with trivial exceptions, these are symmetric groups. With this
knowledge at hand, it was then shown in \cite{YDG2} that the word problem of $\ig{\cE_{\mathcal{T}_n}}$ is decidable for all finite $n$.
Furthermore, it can be amply seen from \cite{Do2} (see Theorem 3.9 and Theorem 2.4, supplemented by remarks preceding Theorem 3.6) that
in the case when the word problem of $\ig{\cE}$ is algorithmically soluble for a finite biorder $\cE$, the only real obstacle towards 
the goal of routinely writing, say, a GAP code \cite{GAP} implementing this word problem is the knowledge of the corresponding 
vertex groups (or, to be more precise, their specific cosets). It was noted in \cite[Remark 3.7]{Do2} that given a finite biorder $\cE$
there exists an algorithm which outputs a finite generating set for any of the required vertex groups (within the direct product of
corresponding maximal subgroups) as well as the necessary coset representatives. However (as we shall see below), the \emph{brute force}
methods for such computations can be very involved. It is thus the purpose of this paper to bypass such methods by providing
combinatorial analysis and arguments sufficient to get hold of these vertex groups and their coset representatives directly. This
reduces the word problem for $\ig{\cE_{\mathcal{T}_n}}$ and the explicit specification of the corresponding algorithm to a 
sequence of standard computational tasks in finite group theory (see Subsection \ref{subsec:comp} below).

The remainder of the paper has two parts. One aims at making this paper reasonably self-contained, and is devoted to the summary of 
all the main notions and results needed to explain the algorithmic problem where the mentioned vertex groups arise, turning out to be 
equivalent to the word problem of semigroups of the form $\ig{\cE}$. In the other half of the paper, in Proposition \ref{pro:deg} and 
Theorem \ref{thm:main} we determine the vertex groups for $\ig{\cE_{\mathcal{T}_n}}$ (with few exceptions that are irrelevant to 
the word problem). Typically, these groups will be subgroups of direct products of finite symmetric groups; they will be closely related 
to a class of permutations preserving certain nice combinatorial configurations.


\section{Preliminaries}\label{sec:prelim}

\subsection{General background}

Throughout we assume familiarity with the basic notions and techniques of semigroup theory, and for these we refer to \cite{HoBook} as a standard textbook 
in the area. In particular, one of the most elementary tools are \emph{Green's relations}: $\R$ relates elements that generate the same principal right
ideal of a semigroups, and $\L,\J$ are respectively the left and the two-sided analogues; also, we have $\H=\R\cap\L$ and $\D=\R\vee\L=\R\circ\L$ (as
$\R\circ\L=\L\circ\R$ holds). Furthermore, we have $\H\subseteq \R,\L \subseteq \D \subseteq \J$, and, in general, all of these containments might be proper. 
On the other hand, it should be noted that $\D=\J$ holds in many natural examples of semigroups: for example, this is true for all finite (and more generally 
for all periodic) semigroups. Also, as proved in \cite[Theorem 4.2(5)]{Do2}, $\D=\J$ also holds in $\ig{\cE}$ whenever $\cE$ is a finite biorder.

These definitions naturally give rise to partial orders on the sets of $\R$-/$\L$-/$\J$-classes of a semigroup (and thus to quasi-orders on the semigroup
itself). Namely, for two $\R$-classes we may write $R_a\leq R_b$ (or, alternatively, $a\leq_\R b$) if and only if $aS^1\subseteq bS^1$; in a similar fashion, 
one can order the $\L$-classes and the $\J$-classes, too.

\begin{exa}
As our main concern in this paper is with the biorders of $\T_n$, the finite full transformation monoids, here is the description of Green's relations in 
$\T_X$ (which are valid for a non-empty set $X$ of any cardinality), assuming that the functions in $\T_X$ are acting on $X$ from the right and are thus
composed left-to-right:
\begin{itemize}
\item $f\;\R\; g \text{ if and only if } \ker f = \ker g$;
\item $f\;\L\; g \text{ if and only if } \im f = Xf = Xg = \im g$;
\item $f\;\J\; g \text{ if and only if } \rank f = |\im f| = |\im g| = \rank g$.
\end{itemize}
In addition, we always have $\D=\J$ in $\T_X$.
\end{exa}

\begin{exa}\label{exa:Tn}
In $\T_n$, $n\geq 1$, the $\J$-/$\D$-classes form a chain of length $n$, as transformations are classified by their rank; so, it is convenient to denote these
classes by $D_n,D_{n-1},\dots,D_1$. Here, $D_n\cong \mathbb{S}_n$ is the group of units consisting of all permutations (= transformations of rank $n$), while
at the other extreme, $D_1$ consists of all constant mappings (forming a semigroup of right zeros). The principal factor associated with $D_m$, $2\leq m\leq n$,
is isomorphic to the Rees matrix semigroup $\mathcal{M}^0[\mathbb{S}_m;I_m;\Lambda_m;P^{(m)}]$, where $I_m$ is the collection of all partitions of $[1,n]$ 
into $m$ classes, $\Lambda_m$ is the collection of all $m$-element subsets of $[1,n]$, while the entry $p_{P,A}$ of the sandwich matrix $P^{(m)}$ is obtained 
in the following way. If $A\perp P$ (which means that $A\in \Lambda_m$ is a cross-section of $P\in I_m$) then this entry is set to be the \emph{label} 
$\lambda(P,A)\in\mathbb{S}_m$ of the pair $(P,A)$ \cite{GR2}, computed as described below; otherwise, it is 0. As for the permutation $\lambda(P,A)$, 
assume that $P=\{P_1,\dots,P_m\}$ and $A=\{a_1,\dots,a_m\}$, with indexing done in such a way that $\min P_1<\dots<\min P_m$ and $a_1<\dots<a_m$. Now, 
the assumption $A\perp P$ ensures that each $P$-class contains a unique element of $A$: say, for each $1\leq i\leq m$, we have that $a_{r_i}\in P_i$. 
Then the mapping
$$
\left(\begin{array}{cccc}
1 & 2 & \dots & m \\
r_1 & r_2 & \dots & r_m
\end{array}\right)
$$
is a permutation, and this is precisely $\lambda(P,A)$.
\end{exa}

\subsection{Basic structural facts about $\ig{\cE}$}

Let $\cE$ be a biordered set -- arising from a semigroup $S$, so that $\cE\cong\cE_S$ -- and let $\Psi:\ig{\cE}\to S$ be the homomorphism, mentioned in the
introduction, extending the map $\ol{e}\mapsto e$, $e\in E(S)$. There is a great degree of similarity between certain aspects of the structure of 
$\ig{\cE}$ and $S'=\gen{E(S)}$. Here we list some of them (see \cite{GR1} for references corresponding to individual results):
\begin{itemize}
\item For any $e\in E$, $\Psi$ maps the $\D$-class of $\ol{e}$ in $\ig{\cE}$ precisely onto the $\D$-class of $e$ in $S'$; it is in this sense that we say 
that the regular $\D$-classes in $\ig{\cE}$ and $S'$ are in bijective correspondence and refer to \emph{corresponding} regular $\D$-classes (in $\ig{\cE}$ and 
$S'$, respectively).
\item $\Psi$ maps the $\R$-class of $\ol{e}$ onto the $\R$-class of $e$, the $\L$-class of $\ol{e}$ onto the $\L$-class of $e$.
\item Consequently, the restriction of $\Psi$ to $H_{\ol{e}}$, the maximal subgroup of $\ig{\cE}$ containing the idempotent $\ol{e}$, is a surjective group 
homomorphism onto $H_e$, the maximal subgroup of $S'$ containing $e$. In other words, the maximal subgroup in a regular $\D$-class of $\ig{\cE}$ is a 
pre-image of the maximal subgroup in the corresponding regular $\D$-class of $S'$.
\end{itemize}
The most fundamental result of the seminal paper \cite{GR1} provides a presentation for these maximal subgroups $H_{\ol{e}}$ (based on the structural data 
about $S'$ as input). This presentation if defined on the set of generators $\{f_{i\lambda}:\ (i,\lambda)\in \mathcal{K}\subseteq I\times\Lambda\}$, where 
$I,\Lambda$ are index sets for the collections of $\R$-/$\L$-classes within $D_{\ol{e}}$ (or within $D_e$ in $S'$, which is the same, as just explained), 
and $\mathcal{K}$ is the set of all pairs with the property that the $\H$-class $H_{i\lambda}=R_i\cap L_\lambda$ is a group, i.e.\ that it contains an 
idempotent (again, it is irrelevant whether we are looking at this within $\ig{\cE}$ or $S'$). As shown in \cite[Theorem 3.10]{DGR}, there is an algorithm 
which, given a finite biordered set $\cE$, computes a (finite) presentation for the maximal subgroup $H_{\ol{e}}$ of $\ig{\cE}$.

\subsection{Regular elements in $\ig{\cE}$ and $\D$-fingerprints}

Since $\ig{\cE}$ is defined in terms of a presentation over a generating set $\ol{E}$, every element of $\ig{\cE}$ can be represented by a word from $E^+$,
in the sense of the natural (surjective) homomorphisms $E^+\to\ig{\cE}$ extending the map $e\mapsto\ol{e}$, $e\in E$. So, for every element of $\ig{\cE}$,
there is at least one word over the alphabet $E$ representing it. The problem is -- and this gives rise to the \emph{word problem} -- this representation
is not necessarily unique: there might be multiple ways to represent an element of $\ig{\cE}$. Thus the word problem (in this case for $\ig{\cE}$) 
asks: is there an algorithm which, presented with two words from $E^+$, decides whether they represent the same element of $\ig{\cE}$? Of course, all along 
the way we assume that $\cE$ is a finite biorder.

Also, a relevant algorithmic question is the following one: given a word $\mathbf{w}=e_1\dots e_m$, decide if it represents a \emph{regular} element of $\ig{\cE}$.
A regularity criterion is found in \cite{DGR}, where in Theorem 3.6 it was proved that $\ol{\mathbf{w}}\in\ig{\cE}$ is regular if and only if $\mathbf{w}$ 
contains a letter $e$ (called the \emph{seed}) so that with the corresponding factorisation $\mathbf{w}=\mathbf{u}e\mathbf{v}$ we have $\ol{\mathbf{u}e}\,
\L\,\ol{e}\,\R\,\ol{e\mathbf{v}}$, in which case $\ol{e}\,\D\,\ol{\mathbf{w}}$. In a certain sense, a sort of a converse statement is true as well: whenever 
we have $\mathbf{w}\equiv \mathbf{u}e\mathbf{v}$ such that $\ol{e}\,\D\,\ol{\mathbf{w}}$, then $e$ is necessarily a seed for $\mathbf{w}$, with 
$\ol{\mathbf{w}}$ being a regular element of $\ig{\cE}$, so that $\ol{\mathbf{u}e}\,\L\,\ol{e}\,\R\,\ol{e\mathbf{v}}$. Furthermore, it was then argued in
Theorem 3.7. of the same paper that this criterion can be effectively tested, so that there is an algorithm which establishes regularity of elements
represented by given words.

For words representing regular elements of $\ig{\cE}$, seeds are not necessarily unique. In fact, the main result of \cite{FG} shows not only that 
it might happen that every letter is a seed, but that in fact whenever $\mathbf{w}=e_1\dots e_m$ represents a regular element, then there are $e_1',\dots,e_m'\in E$ 
such that we have $\ol{e_i'}\in D_{\ol{\mathbf{w}}}$ for all $1\leq i\leq m$, and $\ol{\mathbf{w}} = \ol{e_1'\dots e_m'}$ holds in $\ig{\cE}$. In other words,
any word representing a regular element of $\ig{\cE}$ can be rewritten in terms of idempotents all of which belong to the same $\D$-class as the regular
element itself.

Now let $\ol{\mathbf{w}}$ be a regular element. We have already mentioned that $\D=\J$ holds whenever $\cE$ is finite; hence, we have $D_{\ol{\mathbf{w}}}=
J_{\ol{\mathbf{w}}}$, and the principal factor arising from this $\J$-class must be a completely $0$-simple semigroup, as only finitely many idempotents are involved.
So, we can identify this principal factor with $\mathcal{M}^0[G;I,\Lambda;P]$, where $G$ is the maximal subgroup of $\ig{\cE}$ in the regular $\D$-class $D_{\ol{\mathbf{w}}}$
-- a presentation of which is given e.g.\ in \cite[Theorem 4.2]{DGR}, based on the data from an idempotent-generated semigroup $S'$ such that $\cE\cong\cE_{S'}$ --
$I,\Lambda$ are index sets for the corresponding principal factor ($\D$-class) of $S'$, and $P=[f_{i\lambda}^{-1}]_{I\times\Lambda}$ (see \cite{GR1}). Consequently,
it is possible to write
$$
\ol{\mathbf{w}} = (i,g,\lambda)
$$
for some $g\in G$ (written as a word over $f_{i\lambda}$'s) and $i\in I$, $\lambda\in\Lambda$. Furthermore, as shown in \cite[Theorem 4.3]{YDG2}
(and noted in the subsequent Remark 4.4), there is an algorithm which, presented with a finite biorder $\cE$ and a word $\mathbf{w}\in E^+$,
computes $i,\lambda$, and a word representing $g$.

However, in general, a word $\mathbf{u}$ need not to represent a regular element of $\ig{\cE}$. Yet, what we can do in this case is to consider the coarsest 
factorisation $\mathbf{u}\equiv \mathbf{u}_1\dots\mathbf{u}_k$ into subwords such that each factor represents a regular element; that is to say that whenever 
$\ol{\mathbf{u}_i\dots\mathbf{u}_j}$ is regular for some $1\leq i\leq j\leq k$ then necessarily $i=j$. Such a factorisation is in \cite{YDG2} called a 
\emph{minimal r-factorisation}, and it was explained in \cite[Section 3]{YDG2} that, given $\mathbf{u}$, one can always effectively find one. Of course, 
minimal r-factorisations need not to be unique -- there can be others, for the same word $\mathbf{u}$. Furthermore, there might be another word $\mathbf{v}$ 
such that $\ol{\mathbf{u}}=\ol{\mathbf{v}}$ holds in $\ig{\cE}$, and this word might have a host of its own minimal r-factorisations. Nevertheless, a 
striking result was proved in \cite[Theorem 3.4]{YDG2}: if $\mathbf{u},\mathbf{v}\in E^+$ are two words such that $\ol{\mathbf{u}} = \ol{\mathbf{v}}$, 
with minimal r-factorisations $\mathbf{u}=\mathbf{u}_1\dots\mathbf{u}_k$ and $\mathbf{v}\equiv \mathbf{v}_1\dots \mathbf{v}_r$, then necessarily $k=r$ 
and for all $1\leq i\leq k$ we have $\ol{\mathbf{u}_i}\,\D\,\ol{\mathbf{v}_i}$ (in fact, we even have $\ol{\mathbf{u}_1}\,\R\,\ol{\mathbf{v}_1}$ and 
$\ol{\mathbf{u}_k}\,\L\,\ol{\mathbf{v}_k}$). So, in other words, there is a sequence $(D_1,\dots,D_k)$ of regular $\D$-classes of $\ig{\cE}$ which is 
an invariant of an element of $\ig{\cE}$: no matter what word we consider that represents the element in question, and no matter what minimal r-factorisation 
of that word we take, the (regular) elements represented by the factors will, in the given order, belong to these regular $\D$-classes. Later on, in 
\cite[Theorem 4.2(4)]{Do2}, it was proved that the assumption $\ol{\mathbf{u}}\,\D\,\ol{\mathbf{v}}$ already suffices to arrive at the same conclusion.

As already explained above, given a word representing a regular element, there is an algorithmic procedure of transforming it into a triple of the form 
$(i,g,\lambda)$. Thus if we have a general word $\mathbf{w}\in E^+$ representing an element with $\D$-fingerprint $(D_1,\dots,D_m)$, there is a routine way 
to write up this element as a product
$$
\ol{\mathbf{w}} = (i_1,g_1,\lambda_1)\dots(i_m,g_m,\lambda_m),
$$
where $(i_s,g_s,\lambda_s)\in D_s$ for all $1\leq s\leq m$. Hence, solving the word problem in $\ig{\cE}$ (and, more generally, sorting out its basic structure) 
essentially comes down to comparing products of the above form and, in particular, finding a way to establish whether they are equal in $\ig{\cE}$. In the 
following subsection we are going to introduce the main technical vehicle to express succinctly the gist of the word problem for $\ig{\cE}$. This vehicle is 
also useful in characterising the main structural properties of $\ig{\cE}$, such as its Green's relations.

\subsection{Contact graphs, vertex groups, the map $\theta$}\label{subsec:theta}

First of all, let us note that in this subsection and in the remainder of this section the definition of the map $\theta$, as well as the formulation of all
the relevant results, are slightly modified with respect to the original ones (as they appeared in \cite{YDG2,Do2}). This is done in order to avoid the notion 
of dual groups and thus to contribute slightly to the ``aesthetic appeal'' of the approach. However, it is but an easy exercise to see that the two approaches
are completely equivalent.

Let us start with the following setup. Assume we have given a sequence of groups $G_1,\dots,G_m$, $m\geq 2$. Furthermore, assume that for $1\leq k<m$ we have 
given relations
$$
\rho_k \subseteq G_k \times G_{k+1},
$$
as well as two sequences of elements $a_k,b_k\in G_k$, $1\leq k\leq m$. From these data, we define a new relation $\rho\subseteq G_1\times G_m$ 
by setting that $(g,h)\in\rho$ if and only if there exist $x_r\in G_r$, $2\leq r\leq m$, such that
\begin{align*}
(a_1^{-1}gb_1,x_2) &\in \rho_1, \\
(a_2^{-1}x_2b_2,x_3) &\in \rho_2, \\
&\vdots \\
(a_{m-1}^{-1}x_{m-1}b_{m-1},x_m) &\in \rho_{m-1}, \\
a_m^{-1}x_mb_m &= h.
\end{align*}
This is the general setting we are going to use to describe the map $\theta$ associated with a relation obtained in this way from two elements of $\ig{\cE}$ 
of a given $\D$-fingerprint, maximal subgroups of $\D$-classes involved, and very specific relations obtained from group-labelled graphs we are about to describe. 
(By a map associated with a relation $\rho\subseteq X\times Y$ we mean a function $\varphi_\rho:X\to\mathcal{P}(Y)$ defined by $y\in x\varphi_\rho$ if and only if 
$(x,y)\in \rho$. This is then easily extended to a function $\mathcal{P}(X)\to\mathcal{P}(Y)$ by $A\varphi_\rho=\bigcup_{x\in A}x\varphi_\rho$.)

Assume now that the (finite) biorder $\cE$ comes from an idempotent-generated semigroup $S'$, so that (up to isomorphism) $\cE=\cE_{S'}$. Let $D_1,D_2$ be two 
(not necessarily distinct) regular $\D$-classes of $S'$, whose $\R$-/$\L$-classes are indexed by sets $I_1,I_2$ and $\Lambda_1,\Lambda_2$, respectively. We are 
going to define a graph $\mathcal{A}(D_1,D_2)$ on the vertex set $\Lambda_1\times I_2$ whose edges are labelled by elements of the group $G_1\times G_2$, where $G_1,G_2$ 
are the maximal subgroups of $\ig{\cE}$ in its $\D$-classes corresponding to $D_1$ and $D_2$, respectively. This is going to be the \emph{contact graph} of $D_1$
and $D_2$.

To define this graph, a crucial observation is that elements of $\ol{E}$, the idempotents of $\ig{\cE}$, exercise left and right actions by partial transformations 
on index sets $I$ and $\Lambda$, respectively, of $\R$-/$\L$-classes of a regular $\D$-class $D$ of $\ig{\cE}$ (and thus of a corresponding $\D$-class of $S'$). 
Namely, for $i,i'\in I$ we set $\ol{e}\cdot i=i'$ if 
$$
\ol{e}(i,g,\lambda) = (i',g',\lambda')
$$
holds in $\ig{\cE}$ for some $g,g'\in G$, where $G$ is a maximal subgroup contained in $D$ (with generators $f_{i\lambda}$, as described before), and some 
$\lambda,\lambda'\in\Lambda$. As it transpires from \cite[Proposition 4.1]{YDG2}, we then necessarily have $\lambda'=\lambda$ (and the above relation will hold 
whenever $\lambda$ is replaced by any other index from $\Lambda$), and $g'=cg$, where the coefficient $c$ depends solely on $\ol{e}$ and $i$ (but not on $g$ or 
$\lambda$). Similarly, we set $\lambda\cdot\ol{e} = \lambda'$ if 
$$
(i,g,\lambda)\ol{e} = (i,gd,\lambda')
$$
for some (or all) $i\in I$, and some $g,d\in G$ (where again $d$ depends only on $\ol{e}$ and $\lambda$). Clearly, these actions are vacuous (i.e.\ correspond to
empty partial maps) unless $\ol{e}$ comes from a $\D$-class that is $\J$-above $D$. Also, \cite[Proposition 4.1]{YDG2} shows that the two partial maps induced by 
a given idempotent $\ol{e}$ are simultaneously empty or non-empty, and so the non-emptiness of one of them implies the existence of fixed points of the other one,
and vice versa. The same result supplies the exact information about the coefficients $c,d$; namely,
$$
c = f_{i'\lambda_0}f_{i\lambda_0}^{-1},
$$
where $\lambda_0$ is any fixed point of the right action of $\ol{e}$ upon $\Lambda$, $\lambda_0\cdot\ol{e}=\lambda_0$, such that both $H_{i\lambda_0}$ and $H_{i'\lambda_0}$
are groups, i.e.\ contain idempotents $\ol{e_{i\lambda_0}}$ and $\ol{e_{i'\lambda_0}}$, respectively. Two remarks are important here: the existence of such $\lambda_0$ is
guaranteed, by \cite[Proposition 2.2]{DGR}, by the mere existence of fixed points of the right action of $\ol{e}$ upon $\Lambda$ (and, in fact, by the non-emptiness of 
that action); and secondly, the choice of such $\lambda_0$ is irrelevant because if $\mu_0$ is another such fixed point then the presentation of $G$ contains a relation 
of the form  $f_{i'\lambda_0}f_{i\lambda_0}^{-1}=f_{i'\mu_0}f_{i\mu_0}^{-1}$ arising from an \emph{up-down singular square} $(i,i';\lambda_0,\mu_0)$, see \cite{GR1}. 
Analogously,
$$
d=f_{i_0\lambda}^{-1}f_{i_0\lambda'},
$$
where $i_0$ is any fixed point of the left action of $\ol{e}$ upon $I$ such that both $H_{i_0\lambda}$ and $H_{i_0\lambda'}$ contain idempotents.

Returning to the definition of the contact graph $\mathcal{A}(D_1,D_2)$, for any idempotent $\ol{e}\in\ol{E}$ such that $\lambda = \mu\cdot\ol{e}$ and $\ol{e}\cdot i=j$ 
we draw an edge 
$$(\lambda,i)\longrightarrow(\mu,j),$$ 
and label it with $(a,b^{-1})\in G_1\times G_2$, where, assuming that the generators of $G_1$ are written as $f_{i\lambda}^{(1)}$ and the generators of $G_2$ as 
$f_{i\lambda}^{(2)}$,
\begin{equation}\label{eq:a}
a = \left(f_{i_0\lambda}^{(1)}\right)^{-1}f_{i_0\mu}^{(1)}
\end{equation}
for any fixed point $i_0$ of the left action of $\ol{e}$ on $I$ such that both $H_{i_0\lambda}$ and $H_{i_0\lambda'}$ contain idempotents, and
\begin{equation}\label{eq:b}
b=f_{j\lambda_0}^{(2)}\left(f_{i\lambda_0}^{(2)}\right)^{-1}
\end{equation}
for any fixed point $\lambda_0$ of the right action of $\ol{e}$ on $\Lambda$ such that both $H_{i\lambda_0}$ and $H_{i'\lambda_0}$ contain idempotents. In fact, 
this edge can be traversed in the opposite direction, too, with the amendment that its label is then considered to be $(a^{-1},b) = (a,b^{-1})^{-1}$. As is usually 
the case, the label of a walk is the product of labels of edges along that walk. For $(\lambda,i)\in \Lambda_1\times I_2$ we denote by $W_{(\lambda,i)}$ the collection 
of labels of all closed walks based at the vertex $(\lambda,i)$. As noted in \cite[Lemma 3.2]{Do2}, this collection is actually a subgroup of $G_1\times G_2$, and we call it 
the \emph{vertex group} at $(\lambda,i)$.

Now, let us fix a $\D$-fingerprint $(D_1,\dots,D_m)$ in $\ig{\cE}$, and let 
$$
\mathbf{x} = (i_1,a_1,\lambda_1)\dots(i_m,a_m,\lambda_m)
$$
and
$$
\mathbf{y} = (j_1,b_1,\mu_1)\dots(j_m,b_m,\mu_m)
$$
be two elements of $\ig{\cE}$ of this $\D$-fingerprint. For $1\leq k<m$, put
$$
\rho_k = \left\{\begin{array}{ll}
W_{(\lambda_k,i_{k+1})}(g_k,h_k) & \text{if there exists a walk }(\lambda_k,i_{k+1})\leadsto (\mu_k,j_{k+1}), \\
\es & \text{otherwise}.
\end{array}\right.
$$ 
where $W_{(\lambda_k,i_{k+1})}$ is the vertex group of $\mathcal{A}(D_k,D_{k+1})$ at $(\lambda_k,i_{k+1})$, and $(g_k,h_k)$ is the label of any walk $(\lambda_k,i_{k+1})
\leadsto (\mu_k,j_{k+1})$ (it is immaterial which walk we take, for if $(g_k',h_k')$ is the label of another such walk then $(g_k,h_k)(g_k',h_k')^{-1}$ is the label of 
a closed walk based at $(\lambda_k,i_{k+1})$ and so it belongs to $W_{(\lambda_k,i_{k+1})}$, thus yielding the same right coset). As described previously, these parameters 
define a relation between $G_1$ and $G_m$, and the corresponding mapping $\mathcal{P}(G_1)\to\mathcal{P}(G_m)$ is denoted by $(\cdot,\mathbf{x},\mathbf{y})\theta$.

\subsection{Putting it all together}

In this (modified) setting just described, Theorem 3.9 of \cite{Do2}, characterising the word problem of $\ig{\cE}$, reads as follows.

\begin{thm}\label{thm:wp}
$\mathbf{x}=\mathbf{y}$ holds in $\ig{\cE}$ if and only if $i_1=j_1$, $\lambda_m=\mu_m$, and
$$
1\in(\{1\},\mathbf{x},\mathbf{y})\theta.
$$
\end{thm}

Since by \cite[Theorem 3.8(1)]{Do2} if $A\subseteq G_1$ is a coset of a subgroup of $G_1$ (it doesn't matter if it is left or right, because a left coset is of a subgroup 
is always a right coset of a conjugated subgroup), $(A,\mathbf{x},\mathbf{y})\theta$ is either empty or again a coset of a subgroup, it follows that the condition in the 
previous theorem is equivalent to saying that $(\{1\},\mathbf{x},\mathbf{y})\theta$ is a subgroup of $G_m$.

Similarly, by adapting Corollary 4.3 of \cite{Do2}, Green's relations in $\ig{\cE}$ can be expressed in terms of the map $\theta$ as follows.

\begin{thm} Let $\mathbf{x},\mathbf{y}\in\ig{\cE}$. If these elements are not of the same $\D$-fingerprint, they cannot be $\J$-related. Otherwise, if they are,
we have:
\begin{itemize}
\item[(i)] $\mathbf{x}\,\R\,\mathbf{y}$ if and only if $i_1=j_1$ and $(\{1\},\mathbf{x},\mathbf{y})\theta \neq \es$;
\item[(ii)] $\mathbf{x}\,\L\,\mathbf{y}$ if and only if $\lambda_m=\mu_m$ and $1\in (G_1,\mathbf{x},\mathbf{y})\theta$;
\item[(iii)] $\mathbf{x}\,\D\,\mathbf{y}$ if and only if $(G_1,\mathbf{x},\mathbf{y})\theta \neq \es$.
\end{itemize}
\end{thm}

As already mentioned, $\D=\J$ whenever $\cE$ is finite.

The case when $\mathbf{x}=\mathbf{y}$ is particularly interesting. Here we have that $(H,\mathbf{x},\mathbf{x})\theta$ is a subgroup of $G_m$ whenever $H$ is a subgroup of $G_1$. 
In addition, $(\{1\},\mathbf{x},\mathbf{x})\theta$ is a normal subgroup of $(G_1,\mathbf{x},\mathbf{x})\theta$, and the corresponding quotient is isomorphic precisely to the 
Sch\"utzenberger group of the $\H$-class of $\mathbf{x}$ (see \cite{HoBook}). In fact, exactly along the lines of the proof of Proposition 4.6 of \cite{Do2} it can be proved 
that whenever $H,K$ are two subgroups of $G_1$ such that $K$ is normal in $H$, then $(K,\mathbf{x},\mathbf{x})\theta$ is normal in $(H,\mathbf{x},\mathbf{x})\theta$.

\subsection{The word problem of $\ig{\cE}$ reduced to group theory}\label{subsec:comp}

Now, let us take a slightly more detailed look at the process of computing $(Ht,\mathbf{x},\mathbf{y})\theta$, where $H$ is a subgroup of $G_1$ and $t\in G_1$. We do this to 
make it abundantly clear that this process entirely relies on the knowledge of vertex groups of the contact graphs, along with the information about their connected components 
and the choice of coset representatives arising from walks within these components. We remind, once again, that the result is either the empty set or a coset of a subgroup of $G_m$.

We define two sequences of subgroups $H_k,L_k$ of $G_k$, and two sequences of elements $t_k,z_k\in G_k$, $1\leq k\leq m$, in a recursive fashion (more precisely, at some point, 
some of the sets $H_k$ may become empty, at which point all further $L_k$ are also empty, and the definitions of $t_k,z_k$ become irrelevant). First we set $H_1=H$ and $t_1=t$, 
and then, assuming $H_k$ and $t_k$ have already been defined, let
\begin{align*}
L_k &= a_k^{-1}H_ka_k = H_k^{a_k}, \\
z_k &= a_k^{-1}t_kb_k = t_k^{a_k}a_k^{-1}b_k,\\
H_{k+1}t_{k+1} &= (L_kz_k)\varphi_{\rho_k}, 
\end{align*}
where $\rho_k$ is either $W_{(\lambda_k,i_{k+1})}(g_k,h_k)$, the coset of the vertex group, or the empty relation (depending whether a walk $(\lambda_k,i_{k+1})\leadsto(\mu_k,j_{k+1})$ 
exists or not). As already remarked (and shown in \cite[Theorem 3.8]{Do2}), the last of these recurrences makes sense because its right hand side is the second projection of 
the intersection of cosets
\begin{equation}\label{eq:intersect}
(L_k\times G_{k+1})(z_k,1) \cap W_{(\lambda_k,i_{k+1})}(g_k,h_k),
\end{equation}
which itself is either empty, or a coset of $(L_k\times G_{k+1}) \cap W_{(\lambda_k,i_{k+1})}$. Thus, $H_{k+1}$ is either empty, or the second projection of the subgroup 
$$M_k=(L_k\times G_{k+1}) \cap W_{(\lambda_k,i_{k+1})}$$ 
of $G_k\times G_{k+1}$, and in the latter case the second projection in question is indeed a coset of $H_{k+1}$ (with $t_{k+1}$ chosen arbitrarily such that $(\gamma,t_{k+1})$ belongs 
to the intersection \eqref{eq:intersect} for some $\gamma\in G_k$). Finally, from the very definition of $\theta$ it follows that $(Ht,\mathbf{x},\mathbf{y})\theta = L_mz_m$.


\section{Computing the vertex groups for $\ig{\cE_{\T_n}}$}\label{sec:main}

\subsection{General observations}

We start by recalling a very important remark from \cite{GR1} (made at the beginning of Section 3 of that paper) that the action that elements of $\ol{e}\in\ol{E}$ exercise on 
$\H$-classes contained in an $\R$-class of an idempotent $\ol{f}$ in $\ig{\cE}$ (so, an $\R$-class from a regular $\D$-class) is equivalent the the action that elements of $e\in E$
exercise on $\H$-classes of the $\R$-class $R_f$ in an idempotent-generated semigroup $S$ such that $\cE\cong\cE_S$. (An analogous statement is true for $\H$-classes contained in fixed
regular $\L$-class.) This follows from property (IG3) from that paper, which in turn is a consequence of \cite{FG}. This will substantially facilitate our considerations and computations,
as it means that the (partial) action of $\ol{e}$ on the index sets $I,\Lambda$ associated with the principal factor corresponding to a $\D$-class $D$ such that $D\leq D_{\ol{e}}$ can be
``read off'' already from the semigroup $S$ itself. More precisely, we can formalise this via the following statement. 

\begin{lem}
Let $D$ be a regular $\D$-class of $\ig{\cE}$, with $I,\Lambda$ being the index sets of the collections of $\R$- and $\L$-classes, respectively, contained in $D$. Let $\ol{e}\in\ol{E}$.
Then $\ol{e}\cdot i = i'$ (and so $\ol{e}(i,g,\lambda)=(i',h,\lambda)$ holds for some $\lambda\in\Lambda$ and elements $g,h$ of the maximal subgroup in $D$) if and only if $eH_{i\lambda}
=H_{i'\lambda}$ holds in $S$. Similarly, we have $\lambda\cdot\ol{e}=\lambda'$ if and only if $H_{j\lambda}e=H_{j\lambda'}$ holds in $S$ for some $j\in I$.
\end{lem}

We remind to the facts explained in Example \ref{exa:Tn} that in $\T_n$, transformations on the $n$-element set $[1,n]=\{1,\dots,n\}$ are classified into (regular) $\D$-classes $D_m$ 
according to their rank $m$ (the size of their image), and the corresponding index sets are $I_m$, consisting of all partitions of $[1,n]$ into $m$ classes (the kernels of 
transformations), and $\Lambda_m$, consisting of all $m$-element subsets of $[1,n]$ (the images of transformations) . The role of $S$ (with respect to $\ig{\cE_{\T_n}}$) is taken by 
the idempotent-generated subsemigroup of $\T_n$, which is, by the main result of \cite{Ho}, just $\T_n$ stripped from all the non-trivial permutations, $(\T_n\setminus\mathbb{S}_n)\cup
\{\id_n\}$. As for the maximal subgroups \cite{GR2}, they are the same (in regular $\D$-classes of $\ig{\cE_{\T_n}}$ and $\T_n$), the symmetric group $\mathbb{S}_m$, whenever $m\leq n-2$. 
The only differences arise when $m=n$, when, of course, in $\ig{\cE_{\T_n}}$ the corresponding maximal subgroup is trivial, and when $m=n-1$ when the maximal subgroup in $\ig{\cE_{\T_n}}$
is free of rank $\binom{n}{2}-1$ (and not $\mathbb{S}_{n-1}$). Therefore, a typical regular element of $\ig{\cE_{\T_n}}$ from $\ol{D}_m=D_m\Psi$ may be written as a triple $(P,g,A)$,
where $P$ is a partition of $[1,n]$ into $m$ classes, $A$ and $m$-element subset of $[1,n]$ and $g$ a member of the maximal subgroup contained in $\ol{D}_m$.

We begin with what is essentially a restatement of the previous lemma in the context of $\T_n$ and $\ig{\cE_{\T_n}}$, and its proof is an easy exercise for the reader. For a subset 
$A\subseteq [1,n]$ and a partition $P$ of $[1,n]$ we say that $A$ \emph{saturates} $P$ if every $P$-class contains at least one element of $A$. Also, we say that $P$ \emph{separates} 
$A$ if every $P$-class contains at most one element of $A$. (Clearly, $A\perp P$ if and only if both $A$ saturates $P$ and is separated by $P$.)

\begin{lem}\label{lem:action} 
Let $A,B$ be $m$-element subsets of $[1,n]$ and $P,Q$ partitions of $[1,n]$ into $r$ classes. Let $e$ be an idempotent transformation on $[1,n]$.
\begin{enumerate}
\item $A=B\cdot\ol{e}$ exists if and only if $\ker e$ separates $B$ (in which case $A=Be$).
\item $Q = \ol{e}\cdot P$ exists if and only if $\im e$ saturates $P$ (in which case the classes of $Q$ are the inverse images of those of $P$ under $e$, each of which being a union of certain
$(\ker e)$-classes).
\end{enumerate}
\end{lem}

Besides supplying essential information about the edges in the contact graph $\mathcal{A}(\ol{D}_m,\ol{D}_r)$, this enables us to formulate a regularity criterion within $\ig{\cE_{\T_n}}$ that 
provides us with information beyond that following from \cite[Lemma 6.2]{YDG2}. 

\begin{lem}\label{lem:2reg}
Let $(P,g,A)\in\ol{D}_m$ and $(P',g',A')\in\ol{D}_r$ be two regular elements of $\ig{\cE_{\T_n}}$. Then the product $(P,g,A)(P',g',A')$ is regular if and only if 
\begin{itemize}
\item either $m\geq r$ and $A$ saturates $P'$, or
\item $m\leq r$ and $P'$ separates $A$.
\end{itemize}
\end{lem}

\begin{proof}
($\Ra$) Assume that $(P,g,A)(P',g',A')$ is regular and that $m\geq r$. As shown in the proof of \cite[Proposition 4.1]{YDG2}, and following from \cite{FG}, each of the elements can
be rewritten as a product of idempotents from their own $\D$-classes: so, there are $\ol{e_1},\dots,\ol{e_s}\in\ol{D}_m$ and $\ol{f_1},\dots,\ol{f_t}\in\ol{D}_r$ such that 
\begin{align*}
(P,g,A) &= \ol{e_1}\dots \ol{e_s}, \\
(P',g',A') & = \ol{f_1}\dots\ol{f_t}.
\end{align*}
Now, since the product $\ol{e_1}\dots\ol{e_s}\ol{f_1}\dots\ol{f_t}$ is regular, it must contain a seed; since $r\leq m$, this must be some (and thus any) of the $f_1,\dots,f_t$. In
particular, $f_1$ is a seed letter. But then, by \cite[Remark 2.6]{YDG2}, $\ol{e_sf_1}\,\L\,\ol{f_1}$. On the other hand, $\ol{f_1}\,\R\,\ol{f_1}\dots\ol{f_t}=(P',g',A')$, which
means that $\ker f_1=P'$ and so $\ol{f_1}$ has a representation of the form $(P',h_1,B)$ (for some group element $h_1$ and an $r$-element subset $B$). It follows that 
$\ol{e_s}\cdot P'$ exists, which by the previous lemma means that $\im e_s$ saturates $P'$. However, $\ol{e_1}\dots\ol{e_s}\,\L\,\ol{e_s}$, so $\im e_s = A$, and the claim
follows. We argue in a very similar fashion when $m\leq r$.

($\Leftarrow$) Assume that $m\geq r$ and that $A$ saturates $P'$; the other case is handled analogously. As in the previous part of the proof, each of $(P,g,A)$, $(P',g',A')$ can be
written in $\ig{\cE_{\T_n}}$ as a product of idempotents from their $\D$-classes, just as above. So, if we write $\ol{e_i} = (P_i,g_i,A_i)$ for $1\leq i\leq s$ and $\ol{f_j} = 
(Q_j,h_j,B_j)$ for $1\leq j\leq t$, then by \cite[Lemma 6.2]{YDG2} we must have that $A_i\perp P_{i+1}$ for all $i<s$ and $B_j\perp Q_{j+1}$ for all $j<t$. Also, $A_s=A$ and
$Q_1=P'$. Since the former saturates the latter, by the previous lemma, $\ol{e_s}(P',h_1,B_1)=(P'',h_1',B_1)$, where the classes of $P''=\ol{e_s}\cdot P'$ arise as unions of
$P_s$-classes. Hence, $A_{s-1}$ saturates $P''$. Proceeding in this fashion, we conclude that $(P,g,A)(P',g',A')\in\ol{D}_r$, a regular element of $\ig{\cE_{\T_n}}$. 
\end{proof}

\subsection{The connected components of contact graphs}\label{subsec:conn}

As is well-known \cite{Ho,HoBook}, the idempotent-generated submonoid of $\T_n$ (also called the \emph{singular part} of $\T_n$ and denoted by $\mathrm{Sing}(\T_n)$) is generated
solely by the idempotent transformations of rank $n-1$. These are of the form $\eps_{ij}$ for $1\leq i\neq j\leq n$, so that
$$
k\eps_{ij} = \left\{\begin{array}{ll}
k & k\neq j, \\
i & k=j.
\end{array}\right.
$$
There are $2\binom{n}{2}$ such idempotents, two per each $\R$-class, $n-1$ of them in each of the $\L$-classes. In particular, every non-identity idempotent transformation in 
$\T_n$ can be expressed as a product of rank $n-1$ idempotents. This is reflected in contact graphs in the following way.

\begin{lem}
Let $m,r\leq n-1$ and let $(A,P)$, $(B,Q)$ be two vertices in the contact graph $\mathcal{A}(\ol{D}_m,\ol{D}_r)$ of $\ig{\cE_{\T_n}}$. Then there exists an edge in this graph 
$(A,P)\longrightarrow (B,Q)$ labelled by $e\in E(\T_n)$ if and only if there exist a sequence of vertices $(A_1,P_1),\dots,(A_{k-1},P_{k-1})$ and edges
$$
(A,P)\longrightarrow (A_1,P_1)\longrightarrow \dots \longrightarrow (A_{k-1},P_{k-1})\longrightarrow (B,Q)
$$
labelled, respectively, by $\eps_{i_kj_k},\dots,\eps_{i_1j_1}$ such that $e=\eps_{i_1j_1}\dots\eps_{i_kj_k}$ holds in $\T_n$.
\end{lem}

\begin{proof}
Assume first that $\mathcal{A}(\ol{D}_m,\ol{D}_r)$ contains an edge $(A,P)\longrightarrow (B,Q)$ labelled by $e\in E(\T_n)$. As we may
safely assume that $e\neq\id_n$, by the main result of \cite{Ho}, $e$ can be written as a product of rank $n-1$ idempotents, say
$e=\eps_{i_1j_1}\dots\eps_{i_kj_k}$. Now define the following sequence of subsets of $[1,n]$:
\begin{align*}
A_1 &= B\cdot\ol{\eps_{i_1j_1}}\cdots\ol{\eps_{i_{k-1}j_{k-1}}}, \\
A_2 &= B\cdot\ol{\eps_{i_1j_1}}\cdots\ol{\eps_{i_{k-2}j_{k-2}}}, \\
&\vdots \\
A_{k-1} &= B\cdot\ol{\eps_{i_1j_1}},
\end{align*}
and partitions of $[1,n]$:
\begin{align*}
P_1 &= \ol{\eps_{i_kj_k}}\cdot P, \\
P_2 &= \ol{\eps_{i_{k-1}j_{k-1}}}\cdot\ol{\eps_{i_kj_k}}\cdot P, \\
&\vdots \\
P_{k-1} &= \ol{\eps_{i_2j_2}}\cdots\ol{\eps_{i_kj_k}}\cdot P.
\end{align*}
It takes only a short reflection (upon multiple applications of Lemma \ref{lem:action}) to see that all of these sets and partitions exist,
in the sense that they are all $m$-element subsets and $r$-element paritions of $[1,n]$, respectively. Also, it is now straightforward
to see that we have $A=A_1\cdot\ol{\eps_{i_kj_k}}$ and $\ol{\eps_{i_kj_k}}\cdot P=P_1$, witnessing the existence of an edge $(A,P)
\longrightarrow(A_1,P_1)$ labelled by $\eps_{i_kj_k}$. Furthermore, for all $1\leq s\leq k-2$ we have $A_s=A_{s+1}\cdot
\ol{\eps_{i_{k-s}j_{k-s}}}$ and $\ol{\eps_{i_{k-s}j_{k-s}}}\cdot P_s=P_{s+1}$, showing that there is an edge $(A_s,P_s)\longrightarrow
(A_{s+1},P_{s+1})$ labelled by $\ol{\eps_{i_{k-s}j_{k-s}}}$. Finally, the fact that $A_{k-1}=B\cdot\ol{\eps_{i_1j_1}}$ and
$\ol{\eps_{i_1j_1}}\cdot P_{k-1}=Q$ verifies the existence of an edge $(A_{k-1},P_{k-1})\longrightarrow(B,Q)$ labelled by 
$\ol{\eps_{i_1j_1}}$.

Conversely, assume that there is a walk
$$
(A,P)\longrightarrow (A_1,P_1)\longrightarrow \dots \longrightarrow (A_{k-1},P_{k-1})\longrightarrow (B,Q)
$$
where the edges are labelled, respectively, by $\eps_{i_kj_k},\dots,\eps_{i_1j_1}$, such that the product 
$e=\eps_{i_1j_1}\dots\eps_{i_kj_k}$ is an idempotent in $\T_n$. Then we have $A=A_1\cdot\ol{\eps_{i_kj_k}}$ and 
$\ol{\eps_{i_kj_k}}\cdot P=P_1$, as well as $A_s=A_{s+1}\cdot\ol{\eps_{i_{k-s}j_{k-s}}}$ and $\ol{\eps_{i_{k-s}j_{k-s}}}\cdot P_s=P_{s+1}$
for $1\leq s\leq k-2$, and $A_{k-1}=B\cdot\ol{\eps_{i_1j_1}}$ and $\ol{\eps_{i_1j_1}}\cdot P_{k-1}=Q$. The cumulative effect of these
equalities is that we have $A=B\cdot\ol{e}$ and $\ol{e}\cdot P = Q$, thus proving the existence of an edge $(A,P)\longrightarrow (B,Q)$
labelled by $e$.
\end{proof}

A direct consequence of the previous lemma is that, if we wish to investigate the connected components of contact graphs (which is the purpose of this subsection) it completely 
suffices to focus solely on edges labelled by rank $n-1$ idempotents. So, let us stop for a moment to take a closer look what does it means that we have an edge $(A,P)\longrightarrow 
(B,Q)$ labelled, say, by $\eps_{ij}$.

Indeed, then we have $A=B\eps_{ij}$ and $\ol{\eps_{ij}}\cdot P=Q$. As for the first equality, we have two cases to discuss. First, if $i\in B$, then necessarily $j\not\in B$ 
(for otherwise it would follow that $|B|>|A|$). However, in such a case it follows that $A=B$. Otherwise, $i\not\in B$. Now, if $j\not\in B$ also, then again $A=B$; however, 
if $j$ does belong to $B$, then $A$ is obtained from $B$ by removing $j$ from it and replacing it with $i$: $A=(B\setminus\{j\})\cup\{i\}$. In other words, $B=(A\setminus\{i\})\cup\{j\}$. 

Now let us analyse the second equation. If $i$ and $j$ belong to the same $P$-class, then it is clear that $P=Q$. Otherwise, $i\in P_{k_i}$ and $j\in P_{k_j}$ for some indices
$k_i\neq k_j$. Then the partition $Q$ is obtained from $P$ by adding $j$ to the class $P_{k_i}$ (as both $i,j$ belong to the inverse image of $i$ under $\eps_{ij}$), and consequently,
by removing $j$ from $P_{k_j}$. However, for this latter to be possible (i.e.\ that the described operation does not change the rank of the partition), we must have $|P_{k_j}|>1$,
that is, $P_{k_j}$ must contain at least one additional element except $j$; in other words, $P_{k_j}=\{j\}$ would exclude the possibility of existence of the edge we are considering.

Therefore, we can sum up that traversing an edge labelled by a rank $n-1$ idempotent $\eps_{ij}$ originating from a vertex $(A,P)$ amounts to performing the following ``elementary
step'':
\begin{itemize}
\item Pick $j\not\in A$ \emph{not} comprising a singleton class of $P$, and move it from its $P$-class to another one (or possibly the same one), say $P_s$.
\item If you wish, remove an element $i\in A\cap P_s$ from $A$ and replace it by $j$.
\end{itemize}
Of course, there is also the ``reverse step'', against the arrow, which goes as follows:
\begin{itemize}
\item Remove an element $j\in B$ from $B$ and replace it by (possibly the same) element $i$ from the same $Q$-class.
\item If the previous step is performed with $i\neq j$, then, if you wish, move $j$ from its $Q$-class to another one.
\end{itemize}
These conditions can be made even more compact by saying that moving forth and back along the edges labelled by rank $n-1$ idempotents allows us to do the following two types of moves
with combinatorial ``subset-partition structures'' of the form $(A,P)$, where $A$ is an $m$-element subset of $[1,n]$ and $P$ is a partition of $[1,n]$ into $r$ pieces:
\begin{itemize}
\item[(1)] Move around the subset elements within the given partition class.
\item[(2)] Remove a point currently not belonging to the subset from a non-singleton partition class and add it to another one.
\end{itemize}
(Some rank $n-1$ idempotent allow moves of type (1) and (2) simultaneously, but the same effect can be also achieved by traversing two consecutive edges as well.)

With this in mind, we can now proceed to formulate the connectedness criterion in $\mathcal{A}(\ol{D}_m,\ol{D}_r)$. For a pair $(A,P)$, $|A|=m$, $|P|=r$, we call its \emph{type} the
sequence of numbers $|A\cap P_1|,\dots,|A\cap P_r|$ sorted in non-increasing order; for example, if $n=9$, $A=\{1,3,5,7\}$ and $P=\{\{1,2,6\},\{3,5,7,9\},\{4,8\}\}$, then the type
of $(A,P)$ is $(3,1,0)$. When $(A,P)$ and $(B,Q)$ are of the same type, we say that they are \emph{homeomorphic} and write $(A,P)\sim(B,Q)$. This is the same as saying that there
is a pair of bijections $\phi:A\to B$ and $\psi:P\to Q$ such that for all $1\leq i\leq m$ and $1\leq j\leq r$ we have
$$
a_i\in P_j \quad \text{if and only if} \quad a_i\phi\in P_j\psi.
$$
The pair $(\phi,\psi)$ is then called a \emph{homeomorphism} between $(A,P)$ and $(B,Q)$. 

A pair $(A,P)$ is called \emph{stationary} if all $P$-classes containing elements not in $A$ are singletons.

\begin{pro}
Let $m,r\leq n-1$. Two different vertices $(A,P)$ and $(B,Q)$ of the graph $\mathcal{A}(\ol{D}_m,\ol{D}_r)$ are connected if and only if they are homeomorphic and none of them 
is stationary.
\end{pro}

\begin{proof}
The direct implication of this proposition is immediately clear. Namely, if $(A,P)$ and $(B,Q)$ are indeed connected, then there is a sequence of edges connecting them, each of which 
is labelled by a rank $n-1$ idempotent. Hence, $(B,Q)$ can be obtained from $(A,P)$ by performing a sequence of steps of the type (1) and (2) above. However, note that none of these
steps can change the type of a pair to which it is applied. Thus $(A,P)$ and $(B,Q)$ must be of the same type. Furthermore, none of them are stationary, for otherwise it is clear
that it would not be possible to apply any of the steps (1) or (2) to a stationary pair in a nontrivial fashion.

Conversely, assume that $(A,P)\sim(B,Q)$ and that none of these two pairs is stationary. Upon fixing a homeomorphism $(\phi,\psi):(A,P)\mapsto (B,Q)$, we are going to describe a
sequence of steps (1),(2) that turns $(A,P)$ into $(B,Q)$. Our first aim is to describe a process that, whenever $a\phi\neq a$, ``moves'' the point $a\in A$ to $a\phi$. We have 
three cases  to consider; throughout, we use the assumption that these pairs are non-stationary. For a non-image point $x\in [1,n]\setminus A$ we are going to use the term \emph{free}
if its current partition class is not a singleton (so both $(A,P)$ and $(B,Q)$ have at least one free point each).
\begin{itemize}
\item \emph{$a\phi$ is also in $A$.} Let $x$ be a free point in the current subset-partition structure. Remove it from its partition class and add it to the class containing $a\phi$
(this is a move of type (2)). Then apply a step of type (1) to remove $a\phi$ from the current subset and replace it by $x$. Now, $a\phi$ becomes a free point, so apply a step of 
type (2) to add it to the partition class containing $a$. Finally, apply (1) to remove $a$ from the subset and add $a\phi$. (Note that this makes $a$, at the moment, a free point.)
\item \emph{$a\phi\not\in A$ is a free point.} Remove $a\phi$ from its partition class and add it to the class containing $a$ (this is a step of type (2)). Then (by applying (1))
remove $a$ from the subset and add $a\phi$ to it. (Once again, this makes $a$ a free point.)
\item \emph{$a\phi\not\in A$ currently comprises a singleton partition class.} Take a free point $x$, remove it from its current partition class, and, by applying (2), add it to
$a\phi$, thus making it a class consisting of two free points. This creates a situation from the previous case, so proceed accordingly. 
\end{itemize}
Note that one such step creates a new pair $(A',P')$ (where $A'=(A\setminus\{a\})\cup\{x\}$ in the first case and $A'=(A\setminus\{a\})\cup\{a\phi\}$ in the other two), which is
of course still homeomorphic to $(B,Q)$, via $(\phi',\psi')$ such that 
$$
y\phi'=\left\{\begin{array}{ll}
a\phi & y=a\phi, \\
(a\phi)\phi & y=x, \\
y\phi & y\not\in\{a\phi,x\}
\end{array}\right.
$$
in the first case, and 
$$
y\phi'=\left\{\begin{array}{ll}
a\phi & y=a\phi, \\
y\phi & y\neq a\phi
\end{array}\right.
$$
in the other two. In any case, $\phi'$ has more fixed points than $\phi$. Therefore, either by employing an inductive argument, or by simply iterating the step described above
(applying it now to $(A',P')$ etc.), we arrive at the conclusion that $(A,P)$ can be transformed, by a series of applications of (1) and (2), into  a subset-partition pair
of the form $(B,Q')$ (where the partition $Q'$ is possibly different from $Q$, but induces the same partition on $B$ as $Q$ does).

So, now it remains to argue that we can use steps of type (2) (moving around points not belonging to $B$) in order to transform $Q'$ into $Q$. To be more precise, assume that 
$B=B_1\cup\dots\cup B_t$ is the partition induced on $B$ both by $Q$ and $Q'$ (so that both of the latter contain $r-t$ classes not intersecting $B$). Furthermore, let
$$
B_1\cup B_1',\dots,B_t\cup B_t',C_1',\dots,C_{r-t}'
$$
be the partition classes of $Q'$, while the classes of $Q$ are
$$
B_1\cup B_1'',\dots,B_t\cup B_t'',C_1,\dots,C_{r-t},
$$
with $B_k',B_k'',C_l,C_l'\subseteq [1,n]\setminus B$, $1\leq k\leq t$, $1\leq l\leq r-t$. To prove the required assertion, since all steps involved are reversible, we are going to 
show that both $(B,Q)$ and $(B,Q')$ can be transformed into a fixed pair $(B,Q'')$, where the classes of the partition $Q''$ are
$$
B_1\cup X,B_2,\dots,B_t,\{x_1\},\dots,\{x_{r-t}\},
$$
where $x_1,\dots,x_{r-t}$ are some arbitrary but fixed elements of $[1,n]\setminus B$ and $X=[1,n]\setminus(B\cup\{x_1,\dots,x_{r-t}\})$. We show this for $(B,Q)$, the proof for 
$(B,Q')$ being completely analogous. Now, some of the elements $x_k$, $1\leq k\leq r-t$, may already form singleton classes among $C_1,\dots,C_{r-t}$; without loss of generality 
(and upon renumbering, if necessary) we may assume that $C_k=\{x_k\}$ for all $k<s$ for some $s$. Other classes, not intersecting $B$, namely $C_s,\dots,C_{r-t}$ are either 
not singletons, or are singletons but do not contain any of $x_s,\dots,x_{r-t}$. This, is particular, means that all the elements from the latter list are free in $(B,Q)$. 
However, this very fact allows us to use (2) to remove them from their respective classes and put each $x_l$, $s\leq l\leq r-t$, into $C_l$ (transforming them into $C_l\cup\{x_l\}$). 
But then, at that moment, all members of $C_s,\dots,C_{r-t}$ become free, so all their elements can be sent to the class containing $B_1$ (leaving singleton classes $\{x_s\},\dots,
\{x_{r-t}\}$ behind). The same can be done with all elements of $B_2'',\dots,B_t''$ (as they are obviously free), so we are done.
\end{proof}

This result means that each stationary pair is an isolated vertex in $\mathcal{A}(\ol{D}_m,\ol{D}_r)$, and so we can immediately conclude that its vertex group is trivial. Other,
non-stationary pairs are classified into connected components according to their type.

\subsection{The degenerate cases (i.e.\ involving rank $n-1$)}

It is at this point that we are going explain in full detail the ``with few exceptions'' disclaimer made at the end of the introduction. These exceptions arise because we can
discard some of the pairs/vertices $(A,P)$ of contact graphs, as their vertex groups never appear in the course of computing the map $\theta$ (and thus deciding the word problem
and computing Green's relations in $\ig{\cE_{\T_n}}$); so, some of these ``superfluous'' vertex groups will not be computed here. 

Namely, whenever working with elements of $\ig{\cE_{\T_n}}$, we assume that they are given via some of their minimal r-factorisations (and, as discussed, we can always routinely 
extract at least one such factorisation from a given word over $\ol{E(\T_n)}$). So, a typical such element will be of the form 
$$
\mathbf{x} = (P_1,g_1,A_1)\dots(P_s,g_s,A_s),
$$
and it will be of the $\D$-fingerprint $(\ol{D}_{m_1},\dots,\ol{D}_{m_s})$. Such a setting entails that the product of any two consecutive factors above is a non-regular element
of $\ig{\cE_{\T_n}}$. Bearing in mind Lemma \ref{lem:2reg}, this means that for all pairs $(A_t,P_{t+1})$, $1\leq t<s$, we have that $A_t$ neither saturates $P_{t+1}$, nor is
separated by it. So, whenever we are asked to compute $(\cdot,\mathbf{x},\mathbf{y})\theta$, the process, as can be amply seen, will never involve a vertex group $W_{(A,P)}$
such that $A$ saturates $P$ or is separated by $P$. This motivates a definition: pairs $(A,P)$ with either of these two properties will be called \emph{regular}; otherwise,
they are \emph{non-regular}. 

Therefore, the conclusion is that the vertex groups of regular pairs $(A,P)$ are \emph{completely irrelevant} for the process of computing the map $\theta$; it is only
the vertex groups and their cosets of non-regular pairs that can possibly be interesting for us from the practical point of view. However, then we have the following result.

\begin{pro}\label{pro:deg}
If either $m=n-1$ of $r=n-1$ then a vertex $(A,P)$ in the graph $\mathcal{A}(\ol{D}_m,\ol{D}_r)$ is non-regular if only if it is stationary. Consequently, in such a case, its 
vertex group is trivial. 
\end{pro}

\begin{proof}
First let $m=n-1$. Then $A=[1,n]\setminus\{i\}$ for some $i$. So, the pair $(A,P)$ is regular unless $\{i\}$ is a singleton class in $P$. But this is precisely the case when
$(A,P)$ is stationary under the assumption $m=n-1$. Similarly, assume now that $r=n-1$. Then all classes of $P$ are singletons except one, which contains two elements, say 
$\{i,j\}$. Now the only way $(A,P)$ can be non-regular is that both $i$ and $j$ belong to $A$. But this is also precisely the condition that makes all the $P$-classes
not containing elements of $A$ singletons. 
\end{proof}

Hence, if either $m=n-1$ or $r=n-1$, the non-regular pairs are all isolated vertices of their corresponding contact graphs. On the other hand, when $m=r=n-1$, all regular pairs 
form a single connected component. Computing $W_{(A,P)}$ for regular pairs $(A,P)$ when one of $m,r$ is equal to $n-1$ would involve dealing with the $2\binom{n}{2}$ generators 
$f_{Q,B}$ (with $|Q|=|B|=n-1$ and $B\perp Q$) of the presentation from \cite{GR2} for the maximal subgroup of $\ol{D}_{n-1}$, delving into the combinatorial conditions which of 
these generators are equal to 1 according to this presentation (these are the only defining relations appearing in that presentation), and then performing tedious yet unnecessary 
computations of subgroups within groups of one of the forms $F_k\times F_k$, $F_k\times\mathbb{S}_r$, and $\mathbb{S}_m\times F_k$ for $k=\binom{n}{2}-1$, whereas, from the 
standpoint of (computational) applications in the required context, the previous proposition is all we need.

So, in the remainder or the paper we may safely assume that $m,r\leq n-2$, which eliminates the appearance of free groups. In that case, we however \emph{will} compute the
vertex groups $W_{(A,P)}$ even for the regular pairs $(A,P)$, for the simple reason that the regular case turns out to be not one bit different from the non-regular one;
said otherwise, (non-)regularity of the pair has no impact whatsoever on proving the general result.

\subsection{Vertex groups for $\mathcal{A}(\ol{D}_m,\ol{D}_r)$ when $m,r\leq n-2$}

We start by discussing the group labels associated with a general edge in $\mathcal{A}(\ol{D}_m,\ol{D}_r)$.

\begin{pro}\label{pro:lab}
Let $A,B$ be $m$-element subsets of $[1,n]$, and let $P,Q$ be partitions of $[1,n]$ into $r$ pieces such that there is an edge in $\mathcal{A}(\ol{D}_m,\ol{D}_r)$ directed from
$(A,P)$ to $(B,Q)$ and labelled by $e\in E$. Then this edge carries the group label $(\pi,\pi')$, where $\pi\in\mathbb{S}_m$ and $\pi'\in\mathbb{S}_r$ are permutations such that
$$
b_{i\pi} e = a_i
$$
holds for all $1\leq i\leq m$ (assuming that $a_1<\dots<a_m$ and $b_1<\dots<b_m$), and
$$
P_j e^{-1} = Q_{j\pi'}
$$
holds for all $1\leq j\leq r$ (assuming that $\min P_1<\dots<\min P_r$ and $\min Q_1<\dots<\min Q_r$).
\end{pro}

\begin{proof}
Let us begin by recalling that, in the general case, the label of an edge within the contact graph of two regular $\D$-classes corresponding to $e\in E$ is $(a,b^{-1})$, where $a,b$
are given by equations \eqref{eq:a} and \eqref{eq:b}. Here, $f_{i\lambda}^{(k)}$, $k=1,2$, are the generators of the maximal subgroups in the two $\D$-classes involved, appearing
in the presentation for these groups as described in \cite[Theorem 5]{GR1}. In the concrete case, for the biorder of $\T_n$, these generators will be of the form $f_{P,A}^{(k)}$,
where for $k=1$ the partition $P$ and the subset $A$ of $[1,n]$ are of cardinality $m$ for $k=1$, and of cardinality $r$ for $k=2$ (in both cases we must have $A\perp P$). As it is
explained at the beginning of \cite[Section 3]{GR2} (outlining the plan of the proof of the main result of that paper), in the end, when the presentation in question (applied to $\T_n$)
is sorted out -- and identified to be a presentation for a symmetric group -- the generator $f_{P,A}^{(k)}$ will represent the label $\lambda(P,A)$, see Example \ref{exa:Tn}, a
permutation in $\mathbb{S}_m$ or $\mathbb{S}_r$, respectively. This fact will be crucially taken into account in the remainder of the proof.

Concerning $\pi$, the first component of the label of the considered edge, let us begin by noting that if $1\leq i,i'\leq m$ are such that $b_{i'}e=a_i$, then $a_ie=a_i$. Therefore,
$a_i$ and $b_{i'}$ belong to the same $(\ker e)$-class, and, furthermore, each $a_i$ is a fixed point of $e$, so $A$ is separated by $\ker e$. These observations suffice to justify
the existence of a partition $P_0$ of rank $m$, coarser than $\ker e$, separating $A$ (and thus $B$). Bearing in mind \eqref{eq:a} and the previous remarks, we now have
$$
\pi = \left(f_{P_0,A}^{(1)}\right)^{-1}f_{P_0,B}^{(1)} = \lambda(P_0,A)^{-1}\lambda(P_0,B).
$$
Assume now that the classes of $P_0$ are $P_0^{(s)}$, $1\leq s\leq m$, so that $\min P_0^{(1)}<\dots<\min P_0^{(m)}$. If we write $\sigma=\lambda(P_0,A)$ and $\tau=\lambda(P_0,B)$,
then $a_{s\sigma},b_{s\tau}\in P_0^{(s)}$ holds for all $1\leq s\leq m$. We have already argued that the action of $e$, restricted to $B$, maps each element of $B$ into the elements
of $A$ in the same $(\ker e)$-class (and thus in the same $P_0$-class). The conclusion is that we have $b_{s\tau}e=a_{s\sigma}$ for all $1\leq s\leq m$. Upon re-indexing $i=s\sigma$,
we arrive at 
$$
b_{i\pi}e = b_{i\sigma^{-1}\tau}e = a_i
$$
holding for all $1\leq i\leq m$ , which is precisely what we wanted to show.

We proceed by discussing the second label, $\pi'$. Now, we have $\ol{e}\cdot P=Q$, and so $\im e$ saturates $P$ by Lemma \ref{lem:action}. If $l$ is an image point of $e$ such that
$l\in P_j$ (for some $1\leq j\leq r$), then $le=l$ and so $l\in P_je^{-1}$ as well, the latter inverse image being equal to $Q_{j'}$ for some $1\leq j'\leq r$. It follows from these
considerations that there is a subset $A_0\subseteq \im e$ which is a joint transversal for both $P$ and $Q$. By \eqref{eq:b},
$$
\pi' = \left(f_{Q,A_0}^{(2)}\left(f_{P,A_0}^{(2)}\right)^{-1}\right)^{-1} = \lambda(P,A_0)\lambda(Q,A_0)^{-1}.
$$
If we write $A_0=\{a_1<\dots<a_r\}$, $\sigma=\lambda(P,A_0)$, and $\tau=\lambda(Q,A_0)$, then we have $a_t\in P_{t\sigma^{-1}}\cap Q_{t\tau^{-1}}$ for all $1\leq t\leq r$. Similarly
as in the previous paragraph, corresponding $P$-classes and $Q$-classes (with respect to the action of $\ol{e}$) are identified by containing the same element of their joint transversal
$A_0$. Hence, we have $\ol{e}\cdot P_{t\sigma^{-1}} = P_{t\sigma^{-1}}e^{-1} = Q_{t\tau^{-1}}$ for all $1\leq t\leq r$. Again, by re-indexing $j=t\sigma^{-1}$, we obtain
$$
P_j e^{-1} = Q_{j\sigma\tau^{-1}} = Q_{j\pi'},
$$
just as required.
\end{proof}

Clearly, whenever we are presented with a pair of permutations $(\pi,\pi')\in\mathbb{S}_m\times\mathbb{S}_r$ and two pairs $(A,P),(B,Q)$ such that $|A|=|B|=m$ and $|P|=|Q|=r$, we
can construct two bijections $\phi_\pi:A\to B$ and $\psi_{\pi'}$ by defining $a_i\phi_\pi = b_{i\pi}$ and $P_j\psi_{\pi'} = Q_{j\pi'}$. (So, bearing in mind the previous proposition,
it that context $\phi_\pi$ will be the inverse of the bijection $e|_B$, while $\psi_{\pi'}$ coincides with the action of the inverse image of $e$.) However, the most interesting 
situation for us is when $(\pi,\pi')$ gives rise to a homeomorphism $(\phi_\pi,\psi_{\pi'}):(A,P)\sim(B,Q)$ (and it is straightforward to see that any homeomorphism arises in this
way, from a pair of permutations). By far the most important such situation for us is described in the following lemma.

\begin{lem}
If $(\pi,\pi')\in\mathbb{S}_m\times\mathbb{S}_r$ labels an edge $(A,P)\longrightarrow(B,Q)$ in $\mathcal{A}(\ol{D}_m,\ol{D}_r)$, then $(\phi_\pi,\psi_{\pi'})$ is a homeomorphism between
the pairs involved. Consequently, the same conclusion holds for the group label of any walk in $\mathcal{A}(\ol{D}_m,\ol{D}_r)$ and its endpoints.
\end{lem}

\begin{proof}
Proving this lemma amounts to showing that the following equivalence holds ($1\leq i\leq m$, $1\leq j\leq r$):
$$
a_i\in P_j\quad\text{if and only if}\quad b_{i\pi}\in Q_{j\pi'}.
$$
However, this is now routine, bearing in mind the previous proposition. Also, the second part of the lemma follows easily from the fact that the group label along any walk is the
product of labels of its edges, as well as the fact that $\phi_{\pi_1\pi_2}=\phi_{\pi_1}\phi_{\pi_2}$ and $\psi_{\pi'_1\pi'_2}=\psi_{\pi'_1}\psi_{\pi'_2}$ holds for any $\pi_1,\pi_2
\in\mathbb{S}_m$, $\pi_1',\pi_2'\in\mathbb{S}_r$.
\end{proof}

In particular, in the case when $(B,Q)=(A,P)$, a homeomorphism $(\phi,\psi)$ of $(A,P)$ to itself is called an \emph{auto-homeomorphism} of $(A,P)$. A direct consequence of the 
previous lemma reads as follows.

\begin{cor}\label{cor:auto}
The group label of any loop (closed walk) based at $(A,P)$ gives rise to an auto-homeomorphism of $(A,P)$.
\end{cor}

Formulated in a descriptive way, it is pretty clear what an auto-homeomorphism of a pair $(A,P)$ does: it permutes the partition classes containing the same number of elements of $A$,
and then establishes bijections between elements of $A$ belonging to the corresponding partition classes. If the auto-homeomorphism in question arose from the pair
$(\pi,\pi')\in\mathbb{S}_m\times\mathbb{S}_r$, then this first permutation is completely determined by $\pi'$: it is only subject to the restriction that we must have $|P_j|=|P_{j\pi'}|$
for all $1\leq j\leq r$. Then, to choose $\pi$ (which is in fact the totality of the bijections between the intersections of $A$ with the classes of $P$), we need to comply with
the condition that $a_{i\pi}\in P_{j\pi'}$ whenever $a_i\in P_j$, for all $1\leq i\leq m$ and $1\leq j\leq r$. So, basically, $\pi$ is a permutation of $A$ preserving a partition 
that is induced on it by $P$. Now let $\AHom(A,P)$ denote the subgroup of $\mathbb{S}_m\times\mathbb{S}_r$ consisting of all pairs $(\pi,\pi')$ inducing an auto-homeomorphism of
$(A,P)$ in the described way (it is routine to show that such pairs indeed form a group). This leads us to the concrete description of this permutation group.

\begin{pro}\label{pro:ahom}
Let $A$ be an $m$-element subset of $[1,n]$, and let $P$ be a partition of $[1,n]$ into $r$ classes, such that $m_s$, $1\leq s\leq k$, denote the distinct sizes of non-empty intersections
$A\cap P_j$, $1\leq j\leq r$, $\mu_s$, $1\leq s\leq k$, denotes the number of these intersections of size $m_s$, and $\nu$ denotes the number of empty intersections. 
\begin{itemize}
\item[(i)] The first projection $\Gamma_1$ of $\AHom(A,P)$ (that is, the range of first components $\pi$) is isomorphic to the direct product of wreath products
$$
(\mathbb{S}_{m_1}\wr \mathbb{S}_{\mu_1})\times \dots \times (\mathbb{S}_{m_k}\wr \mathbb{S}_{\mu_k}).
$$
\item[(ii)] The second projection $\Gamma_2$ of $\AHom(A,P)$ (i.e.\ the range of second components $\pi'$) is isomorphic to the direct product
$$
\mathbb{S}_{\mu_1}\times \dots \times \mathbb{S}_{\mu_k}\times \mathbb{S}_\nu.
$$
\item[(iii)] For $\pi\in\Gamma_1$, let $\ol\pi$ be the permutation on the set 
$$J_{(A,P)}=\{j\in[1,r]:\ A\cap P_j\neq\es\}$$ 
uniquely determined by $\pi$ by $j\ol\pi=j'$ if and only if $a_{i\pi}\in P_{j'}$ for some (and thus any) $i\in[1,m]$ such that $a_i\in P_j$. Then $(\pi,\pi')\in\AHom(A,P)$ if and only if 
$$\pi'=\ol\pi\oplus\pi'',$$ 
where $\pi''$ is any permutation of the set $[1,r]\setminus J_{(A,P)}$ (describing the part of $\pi'$ corresponding to the permutation of $P$-classes not intersecting $A$).
\end{itemize}
\end{pro}

\begin{proof}
First of all, if $(\pi,\pi')\in\AHom(A,P)$ then $\phi_\pi$ preserves the partition that $P$ induces on $A$: indeed, $a_i,a_{i'}\in P_j$ for some $1\leq i,i'\leq m$, $1\leq j\leq r$,
if and only if $a_{i\pi},a_{i'\pi}\in P_{j\pi'}$. Furthermore, by the compatibility condition just invoked, we must have $\pi'=\ol\pi\oplus\pi''$ for some permutation $\pi''$ 
of the set $[1,r]\setminus J_{(A,P)}$. Conversely, if $\pi\in\mathbb{S}_m$ is any permutation inducing a $P$-preserving permutation of $A$, then it is straightforward 
to see that 
$$(\pi,\ol\pi\oplus\pi'')\in\AHom(A,P)$$ 
for any permutation $\pi''$ of the set $[1,r]\setminus J_{(A,P)}$. Note that these remarks already show (iii). The statement (i)
also follows immediately, as the structure of the group of partition-preserving permutations is well known, see \cite[Lemma 2.1]{ABMS}.

Also, we already know that any $\pi'\in\Gamma_2$ permutes the partition classes of $P$ whose intersection with $A$ have the same cardinality. Conversely, if $\pi'$ is any such
permutation, it is easy to construct a permutation $\pi\in\Gamma_1$ such that $\pi'=\ol\pi\oplus\pi''$, where $\pi''$ is the restriction of $\pi'$ to the classes not containing
any element of $A$ (for example, let $\pi$ be the union of all monotone bijections $A\cap P_j \mapsto A\cap P_{j\pi'}$, $1\leq j\leq r$). Hence, (ii) follows.
\end{proof}

\begin{rmk}
In the notation introduced in Subsection \ref{subsec:theta} (and then crucially used in Subsection \ref{subsec:comp}), the statement (iii) from the previous proposition
can be expressed as follows. If $\rho=\AHom(A,P)\subseteq \mathbb{S}_m\times\mathbb{S}_r$ is considered as a relation, then 
$$
\pi\varphi_\rho = \{\ol\pi\oplus\sigma:\ \sigma\in\mathbb{S}_{[1,r]\setminus J_{(A,P)}}\}.
$$ 
It is immediately seen that the latter set is just a coset of $\Stab(J_{(A,P)})$, the pointwise stabiliser of $J_{(A,P)}$ (which is isomorphic to $\mathbb{S}_\nu$), corresponding 
e.g.\ to $\ol\pi\oplus\id_{[1,r]\setminus J_{(A,P)}}$.
\end{rmk}

Notice that the previous Corollary \ref{cor:auto} can be now reformulated in the following way.

\begin{lem}
For any vertex $(A,P)$ of the graph $\mathcal{A}(\ol{D}_m,\ol{D}_r)$ we have $W_{(A,P)}\leq\AHom(A,P)$.
\end{lem}

However, our aim is to prove that, unless $(A,P)$ is a stationary pair, equality holds in the previous lemma. This is actually the principal result of this paper.

\begin{thm}\label{thm:main}
If the vertex $(A,P)$ is not stationary in $\mathcal{A}(\ol{D}_m,\ol{D}_r)$, then its vertex group $W_{(A,P)}$ coincides with $\AHom(A,P)$, the auto-homeomorphism group of $(A,P)$.
Otherwise, $W_{(A,P)}$ is trivial.
\end{thm}

Bearing in mind the preceding lemma, the strategy for the proof of this theorem is first to identify the generators of $\AHom(A,P)$ (taking Proposition \ref{pro:ahom} into account),
and then (in the non-stationary case), for each of these generators, constructing a loop in $\mathcal{A}(\ol{D}_m,\ol{D}_r)$ based at $(A,P)$ whose group label is 
precisely the generator in question.

\begin{lem}\label{lem:gens}
Let $A$ be an $m$-element subset of $[1,n]$ and let $P$ be a partition of $[1,n]$ into $r$ classes. Let $\tau_{\alpha,\beta}$ denote the transposition of points $\alpha,\beta$.
Then $\AHom(A,P)$ is generated by the following elements:
\begin{itemize}
\item[(i)] for any two $1\leq i\neq i'\leq m$ such that $a_i,a_{i'}$ belong to the same $P$-class, the pairs 
$$
(\tau_{i,i'},\id_r);
$$
\item[(ii)] for any two $1\leq j\neq j'\leq r$ such that $|P_j|=|P_{j'}|=q\neq 0$, the pairs
$$
(\tau_{i_1,i'_1}\dots\tau_{i_q,i'_q},\tau_{j,j'}),
$$
where $A\cap P_j=\{a_{i_1}<\dots<a_{i_q}\}$ and $A\cap P_{j'}=\{a_{i'_1}<\dots<a_{i'_q}\}$;
\item[(iii)] for any two $1\leq j\neq j'\leq r$ such that $A\cap P_j=A\cap P_{j'}=\es$, the pairs
$$
(\id_m,\tau_{j,j'}).
$$
\end{itemize}
\end{lem}

\begin{proof}
Firstly, it is clear that all the listed pairs indeed belong to $\AHom(A,P)$. Conversely, assume that $(\pi,\pi')\in\AHom(A,P)$. Then, by Proposition \ref{pro:ahom}(iii), for some 
permutation $\pi''$ of $N=[1,r]\setminus J_{(A,P)}$ we have
$$
(\pi,\pi') = (\pi,\ol\pi\oplus\pi'') = (\pi,\ol\pi\oplus\id_N)(\id_m,\id_{J_{(A,P)}}\oplus\pi'').
$$
It is immediately clear that the second factor on the right-hand side is generated by pairs of type (iii). As for the first factor, it is generated by pairs of the form
$$
(\gamma,\ol\gamma\oplus\id_N),
$$
where $\gamma$ runs through the generating set of $\Gamma_1$. This is so because the pairs of the form $(\pi,\ol\pi\oplus\id_N)$ form a subgroup of $\AHom(A,P)$ isomorphic to $\Gamma_1$, 
as the bar mapping is a group homomorphism of $\Gamma_1$ into $\Gamma_2$: $\ol{\pi_1\pi_2}=\ol{\pi_1}\;\ol{\pi_2}$ holds for all $\pi_1,\pi_2\in\Gamma_1$. However, we already know the
structure of $\Gamma_1$, from Proposition \ref{pro:ahom}(i): it is a direct product of wreath products of symmetric groups. The standard knowledge on generating sets of wreath products,
considered as semidirect products, see \cite{RoBook}, immediately implies the result of the lemma.
\end{proof}

We may now proceed by proving our main theorem.

\begin{proof}[Proof of Theorem \ref{thm:main}]
We begin with an observation that will greatly simplify the arguments in the remainder of the proof. Namely, the statement of Proposition \ref{pro:lab} extends to the labels of arbitrary 
walks, in the following sense. The next result is easily verified by repeated applications of this proposition.

\bigskip

\noindent\textit{Claim.} Let $(A,P)$ and $(B,Q)$ be two vertices in $\mathcal{A}(\ol{D}_m,\ol{D}_r)$ connected by a walk
$$
(A,P)\longrightarrow (A_1,P^{(1)})\longrightarrow \dots \longrightarrow (A_{k-1},P^{(k-1)})\longrightarrow (B,Q),
$$
where the edges correspond to idempotent transformations $e_1,\dots,e_k\in E(\T_n)$, with labels $(\pi_s,\pi'_s)$, $1\leq s\leq k$, respectively. Let $f=e_k\dots e_1\in\T_n$. Then
the label $(\pi,\pi') = (\pi_1,\pi'_1)\dots (\pi_k,\pi'_k)$ of this walk is (uniquely) determined by the conditions
$$
b_{i\pi}f = a_i \quad\text{and}\quad P_jf^{-1}=Q_{j\pi'}
$$
for all $1\leq i\leq m$ and $1\leq j\leq r$.

\bigskip

When $(B,Q)=(A,P)$, this applies to closed walks, too. So, our aim is to exhibit, for each generator of $\AHom(A,P)$ listed in Lemma \ref{lem:gens}, a closed walk based
at (a non-stationary pair) $(A,P)$ that corresponds to the considered generator in the sense of the previous claim. This will then complete the proof that $W_{(A,P)}=\AHom(A,P)$.
We opt rather to present the moves corresponding to edges of these walks in a descriptive, combinatorial manner than to write out the sequences of steps (and calculate the labels)
formally -- as this would obscure a great deal the essentially simple ideas behind these constructions. 

(1) \emph{Generators of the form $(\tau_{i,i'},\id_r)$ where $a_i,a_{i'}\in P_j$ for some $i,i',j$.} Here, we should present a sequence of steps (1),(2) -- as specified in 
Subsection \ref{subsec:conn} -- starting and ending with the pair $(A,P)$, such that the resulting transformation $f$ switches $a_i$ and $a_{i'}$ and leaves all the other elements
of $A$ intact. Along the way, we assume that $p\in[1,n]\setminus A$ is a free point for $(A,P)$ (such point exists as $(A,P)$ is assumed to be not stationary). To allow easier
tracking of the movement of points, let us call the initial point $a_i$ \emph{red}, and $a_{i'}$ \emph{blue}. Since $p$ is free, we can remove it from its $P$-class, say $P_{j'}$ 
and add it to $P_j$ (unless it is already there, in case $j'=j$) -- this is a move of type (2). Now we can use $p$ to perform the switch within $P_j$ by a sequence of moves of
type (1): move the red point from $a_i$ to $p$, then the blue point from $a_{i'}$ to $a_i$, and finally the red point from $p$ to $a_{i'}$. At the end, if $j'\neq j$, since $p$
is free at this moment, we may apply (2) to move it back to $P_{j'}$.

(2) \emph{Generators of the form $(\tau_{i_1,i'_1}\dots\tau_{i_q,i'_q},\tau_{j,j'})$ where the classes $P_j,P_{j'}$ are of the same cardinality, and the elements of $A$ contained 
in them are switched in a monotone manner.} In other words, we should find the way to switch the classes $P_j$ and $P_{j'}$ and their elements $a_{i_t},a_{i'_t}$, $1\leq t\leq q$,
respectively. Again, as in the previous case, let $p$ be a free point with respect to $(A,P)$, and, for tracking purposes, let us call the class $P_j$, as well as its elements 
belonging to $A$ \emph{red}, and the class $P_{j'}$ and its elements from $A$ \emph{blue}. We start by adding $p$ to $P_j$ (a move of type (2)). Then, we proceed by moving the
first red point (currently at $a_{i_1}$) to $p$. In this moment, $a_{i_1}$ becomes a free point in the red class, so we may move it (by (2)) to the blue class. After this, we
move the first blue point from $a_{i'}$ to $a_i$. Now, $a_{i'}$ becomes a free point in the blue class, so we move it to the red class, and, subsequently, move the red point
currently at $p$ to $a_{i'}$. The cumulative effect of this part of the process is that $a_i$ became a blue point belonging to the blue class, and $a_{i'}$ a red point belonging
to the red class (and $p$ remained in the red class, being a free point again). But it is now clear that in the same fashion this process can be repeated for the (red-blue) pairs
$(a_{i_2},a_{i'_2}),\dots,(a_{i_q},a_{i'_q})$. At the very end, as $q>0$, all points of the involved classes $P_j,P_{j'}$ not belonging to $A$ are free, so they might be freely
exchanged between the red and blue classes; also, $p$ might be returned to the $P$-class it initially belonged to (as it is free at the end of the described process). Thus we 
again arrive at the pair $(A,P)$ with the only difference that now $P_j$ is blue and $P_{j'}$ red, and the pairs of their $A$-points $a_{i_t},a_{i'_t}$, $1\leq t\leq q$,
exchanged colours. This accounts for the required label.

(3) \emph{Generators of the form $(\id_m,\tau_{j,j'})$ with $A\cap P_j=A\cap P_{j'}=\es$.} Effectively, we need to describe how to switch two $P$-classes (namely $P_j$ and $P_{j'}$) 
not containing any elements  of $A$. If at least one of the $P_j,P_{j'}$ is not a singleton, then this is straightforward: only moves of type (2) will suffice, without involvement 
of any other points from $[1,n]$. Otherwise, if both $P_j,P_{j'}$ are singletons, then we may use the existence of a free point $p\in[1,n]\setminus A$. Namely, we may take it out 
of its current $P$-class and ``lend'' it to $P_j$. This creates a situation from the former case, whence we can now use steps of type (2) to construct a loop which corresponds 
to the switch between $P_j\cup\{p\}$ and $P_{j'}$. At the end, we may take back $p$ to its initial $P$-class, thus completing the loop with the required label.

Our theorem is now proved.
\end{proof}

It remains to comment on the choice of coset representatives $(g_k,h_k)$ in the context of the word problem for $\ig{\cE_{\mathcal{T}_n}}$ and, more generally, in the course
of computing the map $\theta$. As explained in Subsection \ref{subsec:theta}, for these it suffices to choose the label of any walk $(A_k,P_{k+1})\leadsto (B_k,Q_{k+1})$ 
whenever such a walk exists. If it does, we have already argued in this paper that then we must have $(A_k,P_{k+1})\sim (B_k,Q_{k+1})$ and so the label associated to 
any homeomorphism will do. Such a homeomorphism is very easy to compute: for the two considered pairs, one needs to match up partition classes of $P_{k+1}$ and $Q_{k+1}$ 
containing the same number of elements from $A_k$ and $B_k$, respectively, and then to (arbitrarily) choose bijections between the elements of these sets from matching classes.



\begin{thebibliography}{99}
\frenchspacing

\bibitem{ABMS}
J. Ara\'ujo, W. Bentz, J. D. Mitchell, C. Schneider,
The rank of the semigroup of transformations stabilising a partition of a finite set, 
\emph{Math. Proc. Cambridge Philos. Soc.} \textbf{159} (2015), 339--353.

\bibitem{BMM}
M. Brittenham, S. W. Margolis, J. Meakin,
Subgroups of free idempotent generated semigroups need not be free,
\emph{J. Algebra} \textbf{321} (2009), 3026--3042.

\bibitem{YDG1}
Y. Dandan, I. Dolinka, V. Gould,
Free idempotent generated semigroups and endomorphism monoids of free $G$-acts,
\emph{J. Algebra} \textbf{429} (2015), 133--176.

\bibitem{YDG2}
Y. Dandan, I. Dolinka, V. Gould,
A group-theoretical interpretation of the word problem for free idempotent generated semigroups,
\emph{Adv. Math.} \textbf{345} (2019), 998--1041.

\bibitem{YG}
Y. Dandan, V. Gould,
Free idempotent generated semigroups over bands and biordered sets with trivial products,
\emph{Internat. J. Algebra Comput.} \textbf{26} (2016), 473--507.

\bibitem{Do1}
I. Dolinka,
A note on free idempotent generated semigroups over the full monoid of partial transformations,
\emph{Comm. Algebra} \textbf{41} (2013), 565--573.

\bibitem{Do2}
I. Dolinka, 
Free idempotent generated semigroups: The word problem and structure via gain graphs,
\emph{Israel J. Math.} \textbf{245} (2021), 347--387.

\bibitem{DG}
I. Dolinka, R. D. Gray,
Maximal subgroups of free idempotent generated semigroups over the full linear monoid,
\emph{Trans. Amer. Math. Soc.} \textbf{366} (2014), 419--455.

\bibitem{DGR}
I. Dolinka, R. D. Gray, N. Ru\v skuc,
On regularity and the word problem for free idempotent generated semigroups,
\emph{Proc. London Math. Soc. (3)} \textbf{114} (2017), 401--432.

\bibitem{DR}
I. Dolinka, N. Ru\v skuc,
Every group is a maximal subgroup of the free idempotent generated semigroup over a band,
\emph{Internat. J. Algebra Comput.} \textbf{23} (2013), 573--581.

\bibitem{Ea1}
D.~Easdown,
Biordered sets of eventually regular semigroups,
\emph{Proc. London Math. Soc. (3)} \textbf{49} (1984), 483--503.

\bibitem{Ea2}
D.~Easdown,
Biordered sets come from semigroups,
\emph{J. Algebra} \textbf{96} (1985), 581--591.

\bibitem{FG}
D.~G.~Fitz-Gerald,
On inverses of products of idempotents in regular semigroups,
\emph{J. Austral. Math. Soc.} \textbf{13} (1972), 335--337.

\bibitem{GAP}
The GAP Group. 
\emph{GAP -- Groups, Algorithms, and Programming.}

\bibitem{GY}
V. Gould, D. Yang,
Every group is a maximal subgroup of a naturally occurring free idempotents generated semigroup,
\emph{Semigroup Forum} \textbf{89} (2014), 125--134.

\bibitem{GR1}
R. Gray, N. Ru\v skuc,
On maximal subgroups of free idempotent generated semigroups,
\emph{Israel J. Math.} \textbf{189} (2012), 147--176.

\bibitem{GR2}
R. Gray, N. Ru\v skuc,
Maximal subgroups of free idempotent generated semigroups over the full transformation monoid,
\emph{Proc. London Math. Soc. (3)} \textbf{104} (2012), 997--1018.

\bibitem{Ho}
J. M. Howie,
The subsemigroup generated by the idempotents of a full transformation semigroup,
\emph{J. London Math. Soc.} \textbf{41} (1966), 707--716.

\bibitem{HoBook}
J.~M.~Howie,
\emph{Fundamentals of Semigroup Theory},
London Mathematical Society Monographs, New Series, Vol. 12, The Clarendon Press, Oxford University Press, New York, 1995.

\bibitem{McE}
B. McElwee,
Subgroups of the free semigroup on a biordered set in which principal ideals are singletons,
\emph{Comm. Algebra} \textbf{30} (2002), 5513--5519.

\bibitem{Na}
K.~S.~S.~Nambooripad,
Structure of regular semigroups. I.
\emph{Mem. Amer. Math. Soc.} \textbf{22} (1979), no. 224, vii+119 pp.

\bibitem{NP}
K.~S.~S. Nambooripad, F.~Pastijn, 
Subgroups of free idempotent generated regular semigroups, 
\emph{Semigroup Forum} \textbf{21} (1980), 1--7.

\bibitem{Pa}
F. Pastijn, 
The biorder on the partial groupoid of idempotents of a semigroup, 
\emph{J. Algebra} \textbf{65} (1980), 147--187.

\bibitem{RoBook}
J. J. Rotman, 
\emph{An Introduction to the Theory of Groups},
Graduate Texts in Mathematics, Vol. 148, Springer-Verlag, New York, 1995.

\end{thebibliography}
\end{document}